\documentclass[a4paper,12pt,reqno]{amsart}

\usepackage{amsfonts}
\usepackage{amsmath}
\usepackage{amssymb}

\usepackage{mathrsfs}
\usepackage{hyperref}

\numberwithin{equation}{section}

\usepackage{graphicx}

 \newtheorem{thm}{Theorem}[section]
 \newtheorem{cor}[thm]{Corollary}
 \newtheorem{lem}[thm]{Lemma}
 
 \theoremstyle{definition}
 \newtheorem{defn}[thm]{Definition}
 \theoremstyle{remark}
 \newtheorem{rem}[thm]{Remark}
 
 \numberwithin{equation}{section}

\newcommand{\ene}{\mathbb{N}}

\newcommand{\bba}{\mathcal{B}}
\DeclareMathOperator{\Tr}{Tr}

\newcommand{\er}{\mathbb{R}}

\newcommand{\zet}{\mathbb{Z}}

\newcommand{\ern}{{\mathbb{R}}^n}
\newcommand{\ard}{{\mathbb{R}}^{2d}}
\newcommand{\arda}{{\mathbb{R}}^{d}}

\newcommand{\M}{\mathcal{M}}
\newcommand{\Fe}{\mathcal{F}}
\newcommand{\bi}{\begin{itemize}}
\newcommand{\ei}{\end{itemize}}
\newcommand{\be}{\begin{enumerate}}
\newcommand{\ee}{\end{enumerate}}
\newcommand{\beq}{\begin{equation}}
\newcommand{\eq}{\end{equation}}
 
\newcommand{\W}{\mathcal{W}}
\newcommand{\Rev}{\mathcal{R}}

\begin{document}

%
%
%
%
%
%
%
%
%
\title[Approximation property, nuclearity and traces]
 {Approximation property and nuclearity on mixed-norm $L^{p}$,
 modulation and Wiener amalgam spaces}

\author[Julio Delgado]{Julio Delgado}


\address{Department of Mathematics\\
Imperial College London\\
180 Queen's Gate, London SW7 2AZ\\
United Kingdom}

\email{j.delgado@imperial.ac.uk}

\thanks{The first author was supported by the
Leverhulme Research Grant RPG-2014-02.
The second author was partially supported by
 the EPSRC Grant EP/K039407/1.
  The third author was supported in part by NSFC, grant 1171023.
  No new data was collected or generated during the course of the research.
The authors have been also supported by the Sino-UK research project by the British
Council China and the China Scholarship Council.}
\author{Michael Ruzhansky}

\address{Department of Mathematics\\
Imperial College London\\
180 Queen's Gate, London SW7 2AZ\\
United Kingdom}

\email{m.ruzhansky@imperial.ac.uk}
\author{Baoxiang Wang}

\address{LMAM, School of Mathematical Sciences\\
Peking University\\
Beijing 100871\\
China}
\email{wbx@pku.edu.cn}



\date{\today}
\begin{abstract}
In this paper we first prove the metric approximation property for weighted mixed-norm 
$L_w^{(p_1,\dots ,p_n)}$ spaces. Using Gabor frame representation 
this implies that the same property holds 
in weighted modulation and Wiener amalgam spaces. As a consequence,
Grothendieck's theory becomes applicable, and we give criteria for nuclearity
and $r$-nuclearity for operators acting on these space as well as derive the
corresponding trace formulae. Finally, we apply the notion of nuclearity
to functions of the harmonic oscillator on modulation spaces.
\end{abstract}

\maketitle

\section{Introduction}

Approximation properties of Banach spaces constitute the fundamental properties of the
geometry of Banach spaces, see e.g.
Figiel, Johnson and Pelczy\'nski 
\cite{Figiel-et-al:IJM} for a recent review of the subject, as well as
Pietsch's book \cite[Section 5.7.4]{Pietsch:bk-history-2007} for a survey of
different approximation properties and relations among them as well as for the
historical perspective. 

Indeed, one of the importances of this particular property is that once a Banach space is known to have it, the Grothendieck theory of nuclear operators becomes applicable, leading to numerous further developments. Overall, the topic finds itself closely related to a wide range of analysis: spectral analysis, operator theory, functional analysis, harmonic analysis, 
partial differential equations.
On one hand, the question of a space having an approximation property is important for general spaces of functional analysis, see e.g. \cite{ap:sz} or \cite{Szarek:AM-1987}.
On the other hand, it is important to know this also for a range of particular spaces.
For example,  
Alberti, Cs\"ornyei, Pelczy\'nski and Preiss 
\cite{Alberti-Csornyei-Pelczynski-Preiss:BV} 
established the bounded approximation property (BAP) for functions of bounded variations, and
Roginskaya and Wojciechowski \cite{Roginstaya:arxiv-2014}
for Sobolev spaces $W^{1,1}.$

In the paper this property is established for three scales of 
spaces that are of importance in different applications.
First, the mixed Lebesgue spaces provide for a basic tool for harmonic analysis and evolutions partial differential equations (e.g. through Strichartz estimates). The approximation property of such spaces may give rise to an introduction of further spectral methods (following Grothendieck) to questions of harmonic analysis and partial differential equations. Thus, a part of the paper is also devoted to the development of some of these ideas. Second, Wiener amalgam spaces are a central object of the time-frequency analysis, another area with links to several mathematical subjects as well as its applications. Finally, the approximation property in the scale of modulation spaces gives rise to the introduction of further spectral analysis to partial differential equations of very different type -- these spaces become more and more effective (in addition to Besov spaces) in many types of equations including such equations as the Navier-Stokes equation, see e.g. \cite{Iwabuchi:JDE-NS-2010}.

There are other links between this subject and other mathematical areas through the study of Fredholm determinants
(and this is also one of the implications of the paper).
More specifically, determinants of operators of the form $I+A$ are an  important tool in the study of certain differential equations, a known fact that goes back in a rigorous shape to H. Poincar\'e in his work on the Hill's equation \cite{Poincare:1886} where the Banach space $\ell^1$ is relevant. A point of view to define the determinant of $I+A$ consists in considering $A$ as an operator belonging to a class endowed with a trace. This was the idea adopted by Grothendieck in \cite{Grothendieck:Fredholm-BSMF-1956} in the setting of Banach spaces.
In general, there are several approaches to define traces and determinants in the setting of Banach spaces, two of them being that of 
 embedded algebras introduced by Gohberg, Goldberg and Krupnik in \cite{goh:trace}
  and the other one  of operator ideals introduced by Pietsch \cite{piet:book}.
These point of view agree when we consider the ideal of nuclear operators on Banach spaces  in the sense of Ruston-Grothendieck satisfying the approximation property and the underlying Banach space is fixed in the point of view of Pietsch.  
The study of Fredholm determinants is an active field of research, in particular
 due to its applications in the analysis of differential equations,
 see e.g. \cite{Zhao-Barnett,Bothner-Its:CMP-2014,Borodin-Corwin-Remenik,Gesztesy-Latushkin-Zumbrun,McKean:CPAM-2003}.
 The interest in such applications has recently also attracted the attention towards
 the numerical analysis of such determinants.  A systematic study of numerical computations for Fredholm determinants  was initiated by Bornemann \cite{Bornemann:MC-2010}.
  We refer to these papers for further
references and motivations. 

\smallskip
In the present paper we prove the metric approximation
property (which, in particular, implies the bounded approximation property 
because the control of the constant is explicit) 
for weighted mixed-norm
Lebesgue spaces, modulation, and Wiener amalgam spaces, and apply this 
property further to derive spectral information about operators acting on these
spaces.

The modulation spaces were introduced in 1983 by H. Feichtinger \cite{feich:mod} and have been intensively investigated in the last decades. We refer the reader to  the survey \cite{feich:hist} 
by Feichtinger for a historical account of the development of such spaces and a good account of the literature.  We also refer to Gr\"ochenig's book
\cite{groch:bK} for the basic definitions and properties of modulation spaces. 
Modulation spaces start finding numerous applications in various problems in linear
and nonlinear partial differential equations, see \cite{Ruzhansky-Sugimoto-Wang:B-2012}
for a recent survey.

The analysis of the Schatten properties of pseudo-differential operators acting on $L^{2}$
but with symbols of low regularity has been a subject of intensive recent research too, see
e.g. Toft \cite{Toft:Schatten-AGAG-2006,Toft:Schatten-modulation-2008},
Sobolev \cite{sob:sch}, or the authors' papers \cite{dr13a:nuclp,dr:sdk}.
In particular, Schatten properties of pseudo-differential operators with symbols in
modulation spaces have been analysed and established by Toft
\cite{Toft:modul,Toft:modul2}, Gr\"ochenig and Toft \cite{Grochenig-Toft:quasi-Banach},
see also \cite{Grochenig-Toft:pseudos-Toeplitz}.

One purpose of this paper is to analyse the analogous properties of operators but this time
acting on modulation spaces. Since these are in general Banach spaces, the Schatten 
properties are replaced by the notion of nuclearity (or $r$-nuclearity) introduced by
Grothendieck \cite{gro:me}. In order for this theory to become effective we prove that the
modulation spaces (and hence also Wiener amalgam spaces) have the 
approximation property. This is done by proving the same property for weighted mixed-norm 
Lebesgue spaces and then using the Gabor frame description of modulation spaces
reducing them to weighted mixed-norm sequence spaces. Consequently, we derive criteria for
nuclearity and $r$-nuclearity of operators acting on modulation spaces with the 
subsequent trace formulae. The obtained results are applied to study functions of
harmonic oscillator on modulation spaces and the corresponding trace formula of
Lidskii type, relating the operator trace to the sums of eigenvalues for these operators.

\smallskip
To formulate the concepts more precisely, we now recall the notion of modulation spaces and that of nuclear operators on Banach spaces. For a suitable weight $w$ on $\ard$, $1\leq p,q<\infty$ and a window $g\in\mathcal{S}(\arda)$ the modulation space $\M_{w}^{p,q}(\arda)$ consists of the temperate distributions $f\in\mathcal{S} '(\arda)$ such that
\begin{equation}\label{EQ:modul}
\|f\|_{\M_{w}^{p,q}}:=\|V_gf\|_{L_w^{p,q}}:=
\left(\int_{\arda}\left(\int_{\arda}|V_gf(x,\xi)|^pw(x,\xi)^pdx\right)^{\frac qp}
d\xi\right)^{\frac 1q}<\infty ,
\end{equation}
where 
$$
V_gf(x,\xi)=\int_{\arda} f(y)\, \overline{g(y-x)}e^{-iy\cdot \xi} dy
$$ 
denotes the short-time Fourier transform of $f$ with respect to $g$ at the point $(x,\xi)$. The modulation space $\M_{w}^{p,q}(\arda)$ endowed with the above norm becomes a Banach space,
independent of $g\not=0$.

We now recall the required basic conditions on $w$ for the development of the theory of  
modulation spaces $\M_{w}^{p,q}$,
and we refer the reader to the Chapter 11 of \cite{groch:bK} for a detailed exposition. 
A {\em weight function} is a non-negative, locally integrable function on $\ard$. 
%
A weight function $v$ on $\ard$ is called {\em submultiplicative, } if
\beq \label{polyw0} 
v(x+y)\leq v(x)v(y) \mbox{ for all }x,y\in\ard . 
\eq
A weight function $w$ on $\ard$ is  {\em v-moderate, } if
\beq 
w(x+y)\leq v(x)w(y) \mbox{ for all }x,y\in\ard . 
\eq
In particular the weights of polynomial type play an important role. They are of the form
\beq\label{polyw} 
v_s(x,\xi)=(1+|x|^2+|\xi|^2)^{s/2}.
\eq 
The $v_{s}$-moderated weights (for some $s$) are called {\em polynomially moderated}.

\medskip
On the other hand, the approximation property on a Banach space is crucial to define the concept of trace of nuclear operators and in particular for the study of trace formulae such as the 
Grothendieck-Lidskii formula. Let $\mathcal{B}$ be a Banach space, a linear operator $T$ from $\bba$ to $\bba$ is called {\em nuclear} if there exist sequences
$(x_{n}^\prime)\mbox{ in } \bba '$ and $(y_n) \mbox{ in } \bba$ such that
$$
Tx= \sum\limits_{n=1}^{\infty} \left <x,x_{n}'\right>y_n \,\mbox{ and }\,
\sum\limits_{n=1}^{\infty} \|x_{n}'\|_{\bba '}\|y_n\|_{\bba} < \infty.
$$
This definition agrees with the concept of trace class operator in
the setting of Hilbert spaces. The set of nuclear operators from $\bba$ into $\bba$ forms the ideal of nuclear operators $\mathcal{N}(\bba)$ endowed with the norm
\beq\nonumber 
N(T)=\inf\{\sum\limits_{n=1}^{\infty} \|x_{n}'\|_{\bba '}\|y_n\|_{\bba} : 
T=\sum\limits_{n=1}^{\infty} x_{n}'\otimes y_n \}.\eq
It is natural to attempt to define the trace of $T\in\mathcal{N}(\bba)$ by
\begin{equation}\label{EQ:Trace}
\Tr (T):=\sum\limits_{n=1}^{\infty}x_{n}'(y_n),
\end{equation}
where $T=\sum\limits_{n=1}^{\infty}x_{n}'\otimes y_n$ is a
representation of $T$. Grothendieck \cite{gro:me} discovered that $\Tr(T)$ is well defined for 
 all $T\in\mathcal{N}(\bba)$ if and only if $\bba$ has the {\em aproximation property}
  (cf. Pietsch \cite{piet:book} or Defant and Floret \cite{df:tensor}),  i.e. if for every compact
   set $K$ in $\bba$ and for every $\epsilon >0$ there exists $F\in \mathcal{F}(\bba) $ such that
\[\|x-Fx\|<\epsilon\quad \textrm{ for every } x\in K,\]
where we have denoted by $\mathcal{F}(\bba)$ the space of all finite rank bounded linear operators
 on $\bba$. We denote by $\mathcal{L}(\bba)$ the $C^*$-algebra of bounded linear operators on $\bba$. 
There are more related approximation properties, e.g. if in the definition above
 the operator $F$ satisfies $\|F\|\leq 1$ one says that $\bba$ possesses the 
 {\em metric approximation property}.
 This is closely related to the bounded approximation property, see e.g.
Lindenstrauss and Tzafriri \cite[Definition 1.e.11]{Lindenstrauss-Tzafriri:bk}.

\medskip
It is well known that the classical spaces $C(X)$ where $X$ is a compact topological space,
 as well as $L^p(\mu)$ for $1\leq p<\infty$ for any measure $\mu$ satisfy the metric approximation property (cf. Pietsch \cite{piet:book1}).
 In \cite{ap:enflo} Enflo constructed a counterexample to the approximation property in Banach spaces. A more natural counterexample was then found by Szankowski \cite{ap:sz}
who proved that $B(H)$ does not have the approximation property. 

An important feature on Banach  spaces even endowed with the approximation property is that the Lidskii formula does not hold in general for nuclear operators, as it was proved by
Lidskii \cite{li:formula} in Hilbert spaces showing that the operator trace is equal to the
sum of the eigenvalues of the operator counted with multiplicities. 
Thus, in the setting of Banach spaces,
Grothendieck \cite{gro:me}
introduced a more restricted class of operators where Lidskii formula holds, 
this fact motivating the following definition.

Let $\bba$ be a Banach space and $0<r\leq 1$, a linear operator $T$
from $\bba$ into $\bba$ is called {\em r-nuclear} if there exist sequences
$(x_{n}^{\prime})\mbox{ in } \bba' $ and $(y_n) \mbox{ in } \bba$ so that
\beq 
Tx= \sum\limits_{n=1}^{\infty} \left <x,x_{n}'\right>y_n \,\mbox{ and }\,
\sum\limits_{n=1}^{\infty} \|x_{n}'\|^{r}_{\bba'}\|y_n\|^{r}_{\bba} < \infty.\label{rn}
\eq
We associate a quasi-norm $n_r(T)$ by
\[
n_r(T)^r:=\inf\{\sum\limits_{n=1}^{\infty} \|x_{n}'\|^{r}_{\bba'}\|y_n\|^{r}_{\bba}\},
\]
where the infimum is taken over the  representations of $T$ as in \eqref{rn}. 
When $r=1$ the $1$-nuclear operators agree with 
the nuclear operators, in that case this definition agrees with the concept of trace class operator
 in the setting of Hilbert spaces ($\bba=H$). More generally, Oloff proved in \cite{Oloff:pnorm} that the class of $r$-nuclear
operators coincides with the Schatten class $S_{r}(H)$ when $\bba=H$ and 
$0<r\leq 1$. Moreover, Oloff proved that 
\beq\label{olo1}\|T\|_{S_r}=n_r(T),\eq
where $\|\cdot\|_{S_r}$ denotes the classical Schatten quasi-norms in terms of singular values.

In \cite{gro:me} Grothendieck proved that if $T$ is $\frac 23$-nuclear from $\bba$ into $\bba$ for a Banach space $\bba$ endowed with the 
 approximation property, then
\beq\Tr(T)=\sum\limits_{j=1}^{\infty}\lambda_j,\label{lia1}\eq
where $\lambda_j\,\, (j=1,2,\dots)$ are the eigenvalues of $T$ with multiplicities taken into account,
and $\Tr(T)$ is as in \eqref{EQ:Trace}.
Nowadays the formula
 \eqref{lia1} is refered to as Lidskii's formula, proved by V. Lidskii \cite{li:formula}
 in the Hilbert space setting. 
 Grothendieck also established its applications to the distribution of eigenvalues of operators
in Banach spaces. We refer to \cite{dr13a:nuclp} for several conclusions 
in the setting of compact Lie groups
concerning
summability and distribution of eigenvalues of operators on $L^{p}$-spaces once
we have information on their $r$-nuclearity. Kernel conditions on compact manifolds have been investigated in 
\cite{dr:suffkernel}, \cite{dr:sdk}. 

For our purposes it is convenient to consider first mixed-norm spaces. 
In Section \ref{SEC:AP} we establish
 the metric approximation property for the weighted mixed-norm 
 $L_w^{(p_1,\dots ,p_n)}$ spaces, and in Section \ref{SEC:AP-mw}
 for modulation spaces $\M_{w}^{p,q}$
 and Wiener amalgam spaces $\W_{w}^{p,q}$, also recalling the definition of 
 the latter.
In Section \ref{SEC:r-nuc-Lp} we characterise $r$-nuclear operators 
acting between weighted mixed-norm spaces $L^{P}_{w}$.
In Section \ref{SEC:r-nuc-mod} we apply it to the questions of 
$r$-nuclearity and trace formulae in modulation spaces
and functions of the harmonic oscillator in that setting. 

\section{$L_w^{(p_1,\dots ,p_n)}$
has the metric approximation property}
\label{SEC:AP}

In  this section we prove that mixed-norm spaces
$L_w^{(p_1,\dots ,p_n)}$ and consequently the modulation spaces $\M_{w}^{p,q}$ 
satisfy  the metric approximation property. Through the Fourier transform, 
also the Wiener amalgam spaces
$\W^{p,q}_{w}$ will have the same property. 

\medskip
We start the analysis of the approximation property by recalling 
a basic lemma which simplifies the proof 
of the metric approximation property (cf. \cite{piet:book1}, Lemma 10.2.2).

\begin{lem} \label{LEM:approx}
A Banach space $\bba$ satisfies the metric approximation property if, given $x_1,\dots,x_m\in\bba$ and $\epsilon >0$ there exists an operator $F\in\mathcal{F}(\bba)$ such that $\|F\|\leq 1$ and 
\[\|x_i-Fx_i\|\leq\epsilon\,\mbox{ for }i=1,\dots,m.\]
\end{lem}
We shall now establish the metric approximation property for weighted mixed-norm spaces. We first briefly recall its definition and we refer the reader to \cite{bp:mxlp} for the basic properties of these spaces. 

Let $(\Omega_i, S_i,\mu_i)$, for $i=1,\dots,n$, be given $\sigma$-finite measure spaces. We
write $x=(x_{1},\ldots,x_{n})$, and let
$P=(p_1,\dots,p_n)$ a given $n$-tuple with $1\leq p_i<\infty$. We say that $1\leq P<\infty $ if $1\leq p_i<\infty$ for all $i=1,\dots, n$. Let $w$ be a strictly positive measurable function. The  norm $\|\cdot\|_{L_w^P}$
 of a measurable function $f(x_1,\dots,x_n)$ on the corresponding product measure space is defined by
\[\|f\|_{L_w^P}:=\left(\int_{\Omega_n}\cdots\left(\int_{\Omega_2}\left(\int_{\Omega_1}|f(x)|^{p_1}w(x)d\mu_1(x_{1})\right)^{\frac{p_2}{p_1}}d\mu_2(x_{2})\right)^{\frac{p_3}{p_2}}\cdots d\mu_n(x_{n})\right)^{\frac{1}{p_n}}.\]
As should be observed, the order of integration on these spaces is crucial. 
$L_w^P$-spaces endowed with the $\|\cdot\|_{L_w^P}$-norm become Banach spaces and the dual $(L_w^P)'$ of $L_w^P$ is $L_{w^{-1}}^{P'}$, where $P'=(p_1',\dots,p_n')$. In view of our application to the modulation spaces $\M_w^{p,q}$ we will consider in particular the case of the index of the form $(P,Q)=(p_1,\dots,p_d,q_1,\dots,q_d)$ where $p_i=p, q_i=q$ and $\Omega_i=\er$ endowed with the Lebesgue measure. In this case the weight is taken in the form $w=w(x,\xi)$ where 
$x\in\er^d, \xi\in\er^d$, with some special conditions on $w$ that we briefly recall at the end of this section.

 We first establish
 a useful lemma for the proof of the metric approximation property. We will require  some notations. For $n\in\ene$ and $P=(p_1,\dots,p_n)$  we will denote by $\ell^P(I)$ the  $L^P$ mixed-norm space corresponding to $I^n$ where $I$ is a countable set of indices endowed with the counting measure. We note that such $\ell^P$-norm is given by
\[\|h\|_{\ell^P}=\left(\sum\limits_{k_n\in I}\cdots\left(\sum\limits_{k_2\in I}\left(\sum\limits_{k_1\in I}|h(k_1,\dots,k_n)|^{p_1}\right)^{\frac{p_2}{p_1}}\right)^{\frac{p_3}{p_2}}\cdots\right)^{\frac{1}{p_n}}.\]
Given a Banach space $\bba$ and $u\in\bba, z\in\bba '$ we will also denote by $\langle u,z\rangle _{\bba, \bba'}$, or simply by $\langle u,z\rangle $, the valuation $z(u)$. 

\begin{lem}\label{simp1} Let $\bba$ be a Banach space and $1\leq Q<\infty$. Let $I$ be a
 countable set. Let $(u_i)_{i\in I}, (v_i)_{i\in I}$ be sequences in $\bba ', \bba$ respectively such that
\[\|\langle x,u_i\rangle\|_{\ell ^{Q}(I)}, \|\langle v_i,z\rangle\|_{\ell ^{Q'}(I)}\leq 1,\,\,{\mbox { for }} \|x\|_{\bba}, \|z\|_{\bba '}\leq 1.\]
Then the operator $T=\sum\limits_{i\in I}u_i\otimes v_i$ from $\bba$ into $\bba$ is well defined, bounded and satisfies
$\|T\|_{\mathcal{L}(\bba)}\leq 1$.
\end{lem}
\begin{proof} Let $N\subset I$ be a finite subset of $I$. Let us write $T_N:=\sum\limits_{i\in N}u_i\otimes v_i$.
 It is clear that $T_N$ is well defined. Moreover $T_N$ is a bounded finite rank operator.\\
Now, since 
$T_Nx=\sum\limits_{i\in N}\langle x,u_i\rangle v_i$, we observe that for $x\in\bba, z\in\bba '$ such that $\|x\|_{\bba}, \|z\|_{\bba '}\leq 1$ we have 
$$
|\langle T_Nx,z\rangle|\leq \sum\limits_{i\in N}|\langle x,u_i\rangle||\langle v_i,z\rangle|
 \leq \|\langle x,u_i\rangle\|_{\ell ^{Q}}\|\langle v_i,z\rangle\|_{\ell ^{Q'}}
\leq  1.
$$
 We have applied the H\"older inequality in $\ell ^{Q}$ mixed-norm spaces (cf. \cite{bp:mxlp}) for the second inequality. 
Hence $T=\lim\limits_NT_N$ exists in $\mathcal{L}(\bba)$ and $\|T\|_{\mathcal{L}(\bba)}\leq 1$.
\end{proof}
We will apply the lemma above for the particular case of the single index $Q=p_n$ and the Banach space $L^P$. In the rest of
 this section we will assume that 
 the weights $w$ satisfy the following condition for all $x\in\Omega$:
 \beq\label{conw1}w(x_1,\dots,x_n)\leq w_1(x_1)\cdots w_n(x_n),\eq
where $w_j$ is a weight on $\Omega_j$
(i.e. a strictly positive locally integrable function). 
In particular, the condition holds for polynomially moderate
weights on $\ern$ satisfying  for a suitable $n$-tuple $(\beta_1,\dots,\beta_n)$,
the condition
\beq
\label{conw2}
w(x_1,\dots,x_n)\leq \langle x_1\rangle^{\beta_1}\cdots\langle x_n\rangle^{\beta_n},
\eq
where $\langle x_{j}\rangle =1+|x_{j}|$.

\begin{thm} \label{app12} 
The weighted mixed-norm spaces $L_w^P=L_w^{(p_1,\dots ,p_n)}$  with $w$ satisfying \eqref{conw1} have the metric approximation property.
\end{thm}

\begin{proof} (Step 1) We will first prove the metric approximation property for the constant weight $w=1$. 
Let $f_1,\dots,f_N\in L^{(p_1,\dots ,p_n)}$ and let $\epsilon>0$. 
By Lemma \ref{LEM:approx} it is enough to 
construct an operator $L\in\Fe(L^P,L^P)$ such that $\|L\|_{\mathcal{L}(L^P)}\leq 1$ and
\begin{equation}\label{EQ:app}
\|f_i-Lf_i\|_{L^{P}}\leq\epsilon ,\, \mbox{ for }i=1,\dots,N.
\end{equation}
We consider first elementary functions of the form
\[
f_i^0(x_1,\dots,x_n):=\sum\limits_{k_1,\dots,k_n=1}^l\alpha_{k_1,\dots,k_n}^i\prod\limits_{j=1}^n1_{\Omega_{k_j}^j}(x_j),
\]
where the sets $\prod\limits_{j=1}^n\Omega_{k_j}^j$ are disjoint, $1_{\Omega_{k_j}^j}$ denotes the characteristic function of the set $\Omega_{k_j}^j$, 
\[\mu_j(\Omega_{k_j}^j)<\infty\]
and
\[\|f_i-f_i^0\|_{L^{(p_1,\dots , p_n)}}\leq\frac{\epsilon}{2}.\]
For the density of simple functions in $L^P$ see \cite{bp:mxlp}. Since we
 are excluding the index $p=\infty$ in the multi-index $P$, the density always holds in our context. \\
 
We will denote
\[(\prod\limits_{j=1}^n1_{\Omega_{k_j}^j})(x_1,\dots, x_n):=\prod\limits_{j=1}^n1_{\Omega_{k_j}^j}(x_j).\]
We define
\[u_{(k)}:=\frac{\prod\limits_{j=1}^n1_{\Omega_{k_j}^j}}{\prod\limits_{j=1}^n\mu_j(\Omega_{k_j}^j)^{\frac{1}{p_j'}}},\,\,v_{(k)}:=\frac{\prod\limits_{j=1}^n1_{\Omega_{k_j}^j}}{\prod\limits_{j=1}^n\mu_j(\Omega_{k_j}^j)^{\frac{1}{p_j}}}\,, \]
where $(k)=(k_1,\dots,k_n)$.\\

Set
\begin{align*}L:=&\sum\limits_{k_1,\dots,k_n=1}^l\frac{\prod\limits_{j=1}^n1_{\Omega_{k_j}^j}}{\prod\limits_{j=1}^n\mu_j(\Omega_{k_j}^j)^{\frac{1}{p_j'}}}\otimes\frac{\prod\limits_{j=1}^n1_{\Omega_{k_j}^j}}{\prod\limits_{j=1}^n\mu_j(\Omega_{k_j}^j)^{\frac{1}{p_j}}}\\
=&\sum\limits_{(k)=k_1,\dots, k_n=1}^lv_{(k)}\otimes u_{(k)}\\
=&\sum\limits_{(k)=k_1,\dots, k_n=1}^l\frac{1}{\prod\limits_{j=1}^n\mu_j(\Omega_{k_j}^j)}\left(\prod\limits_{j=1}^n1_{\Omega_{k_j}^j}\otimes\prod\limits_{j=1}^n1_{\Omega_{k_j}^j}\right).
\end{align*}
In order to prove that $\|L\|_{{\mathcal{L}(L^P)}}\leq 1$ we will apply 
Lemma \ref{simp1} for $\bba=L^P$, the families $u(k),v(k)$ and $Q=p_n$. The special role of the power $p_n$ will become
 clear later. Let $f\in L^P$, $g\in L^{P'}$ be such that $\|f\|_{L^P}, \|g\|_{L^{P'}}\leq 1$. Then we have to show that
\[\|\langle f,u_{(k)}\rangle\|_{\ell ^{p_n}}\leq 1
\; \textrm{ and } \; \|\langle v_{(k)},g\rangle\|_{\ell ^{p_n'}}\leq 1 . \]
In order to verify the corresponding property for $f\in L^P $, it is enough to consider an  
elementary function $f\in L^P$ such that $\|f\|_{L^P}\leq 1$. The general case follows then by approximation.
 By redefining partitions, we can assume that $f$ can be written in the form
\[f(x_1,\dots,x_n)=\sum\limits_{k_1,\dots,k_n=1}^l\lambda_{k_1,\dots,k_n}\prod\limits_{j=1}^n1_{\Omega_{k_j}^j}(x_j).\]
We note that 
\[\langle f,u_{(k)}\rangle =\lambda_{k_1,\dots,k_n}\prod\limits_{j=1}^n\mu_j(\Omega_{k_j}^j)^{\frac{1}{p_j}},\]
and
\[\|\langle f,u_{(k)}\rangle\|_{\ell ^{p_n}}=\left(\sum\limits_{(k)}|\lambda_{(k)}|^{p_n}\prod\limits_{j=1}^n\mu_j(\Omega_{k_j}^j)^{\frac{p_n}{p_j}}\right)^{\frac{1}{p_n}}.\]
On the other hand, a straightforward but long calculation shows that
\[\|f\|_{L^P(\mu)}=\left(\sum\limits_{k_1,\dots,k_n=1}^l|\lambda_{k_1,\dots,k_n}|^{p_n}\prod\limits_{j=1}^n\mu_j(\Omega_{k_j}^j)^{\frac{p_n}{p_j}}\right)^{\frac{1}{p_n}}.\]
Since $\|f\|_{L^P}\leq 1$, we have shown that $\|\langle f,u_{(k)}\rangle\|_{\ell ^{p_n}}\leq 1$.
 The proof for $\|\langle v_{(k)},g\rangle\|_{\ell ^{p_n'}}$ is similar and we omit it.

\medskip
Therefore $\|L\|_{{\mathcal{L}(L^P)}}\leq 1$.

\medskip
Now we obtain \eqref{EQ:app} in view of
\[\|f_i-Lf_i\|_{L^{P}}\leq\|f_i-f_i^0\|_{L^{P}}+\|Lf_i^0-Lf_i\|_{L^{P}}\leq\epsilon\]
since $Lf_i^0=f_i^0$. Indeed one has
\begin{align*}
L(\prod\limits_{\ell=1}^n1_{\Omega_{k_{\ell}}^{\ell}})=&\sum\limits_{k_1,\dots,k_n=1}^l\left(\frac{\prod\limits_{j=1}^n1_{\Omega_{k_j}^j}}{\prod\limits_{j=1}^n\mu_j(\Omega_{k_j}^j)^{\frac{1}{p_j'}}}\otimes\frac{\prod\limits_{j=1}^n1_{\Omega_{k_j}^j}}{\prod\limits_{j=1}^n\mu_j(\Omega_{k_j}^j)^{\frac{1}{p_j}}}\right)\prod\limits_{\ell=1}^n1_{\Omega_{k_{\ell}}^{\ell}}\\
=&\sum\limits_{k_1,\dots,k_n=1}^l\frac{\left\langle \prod\limits_{\ell=1}^n1_{\Omega_{k_{\ell}}^{\ell}}, \prod\limits_{j=1}^n1_{\Omega_{k_j}^j}\right\rangle _{P',P}}{\prod\limits_{j=1}^n\mu_j(\Omega_{k_j}^j)}\prod\limits_{j=1}^n1_{\Omega_{k_{j}}^{j}}\\
=&\frac{\prod\limits_{\ell=1}^n\mu_{\ell}(\Omega_{k_{\ell}}^{\ell})}{\prod\limits_{\ell=1}^n\mu_{\ell}(\Omega_{k_{\ell}}^{\ell})}\prod\limits_{\ell=1}^n1_{\Omega_{k_{\ell}}^{\ell}}\\
=&\prod\limits_{\ell=1}^n1_{\Omega_{k_{\ell}}^{\ell}}.
\end{align*}
For the third equality we have used the fact that the sets $\prod\limits_{\ell=1}^n1_{\Omega_{k_{\ell}}^{\ell}}$ are disjoint.

(Step 2) We will now prove the metric approximation property for an elementary weight $w$
(in this case we ask for weights to be non-negative and locally integrable). 
We can write such $w$ in the form
\beq w(x_1,\dots,x_n)=\sum\limits_{k_1,\dots,k_n=1}^l\gamma_{k_1,\dots,k_n}\prod\limits_{j=1}^n1_{\Omega_{k_j}^j}(x_j),\label{ew1}\eq
where the sets $\prod\limits_{j=1}^n\Omega_{k_j}^j$ are disjoint, $\mu_j(\Omega_{k_j}^j)<\infty$ and $\gamma_{k_1,\dots,k_n}>0$ for all $(k)$. 

Let $f_1,\dots,f_N\in L_w^{(p_1,\dots ,p_n)}$ and let $\epsilon>0$. 
 With a slight modification of (Step 1) we will find 
 an operator $L\in\Fe(L_w^P,L_w^P)$ such that $\|L\|_{\mathcal{L}(L_w^P)}\leq 1$ and
\begin{equation}\label{EQ:app2}
\|f_i-Lf_i\|_{L_w^{P}}\leq\epsilon ,\, \mbox{ for }i=1,\dots,N.
\end{equation}
We consider again elementary functions, by redefining partitions they can be written in the form
\[f_i^0(x_1,\dots,x_n):=\sum\limits_{k_1,\dots,k_n=1}^l\alpha_{k_1,\dots,k_n}^i\prod\limits_{j=1}^n1_{\Omega_{k_j}^j}(x_j)\]
and
\[\|f_i-f_i^0\|_{L_w^{(p_1,\dots , p_n)}}\leq\frac{\epsilon}{2}.\]

We define
\[u_{(k)}:=\frac{\prod\limits_{j=1}^n\gamma_{k}^{\frac{1}{p_j}}1_{\Omega_{k_j}^j}}{\prod\limits_{j=1}^n\mu_j(\Omega_{k_j}^j)^{\frac{1}{p_j'}}},\,\,v_{(k)}:=\frac{\prod\limits_{j=1}^n\gamma_{k}^{\frac{1}{p_j'}}1_{\Omega_{k_j}^j}}{\prod\limits_{j=1}^n\mu_j(\Omega_{k_j}^j)^{\frac{1}{p_j}}}, \]
and set
\[L:=\sum\limits_{(k)=k_1,\dots, k_n=1}^lv_{(k)}\otimes u_{(k)}.\]
Again, in order to prove that $\|L\|_{{\mathcal{L}(L_w^P)}}\leq 1$ we will apply 
Lemma \ref{simp1} for $\bba=L^P$ and the families $u(k),v(k)$ and $Q=p_n$. 
 We note that $w^{-1}$, defined as $1/w$ on the support of $w$, 
 is also an elementary weight. Let $f\in L_w^P$, $g\in L_{w^{-1}}^{P'}$ 
 be such that $\|f\|_{L_w^P}, \|g\|_{L_{w^{-1}}^{P'}}\leq 1$. 

In order to verify the corresponding property for $f\in L_w^P $, it is enough to consider an  
elementary function $f\in L^P$ such that $\|f\|_{L_w^P}\leq 1$. By redefining partitions, we can assume that $f$ can be written in the form
\[f(x_1,\dots,x_n)=\sum\limits_{k_1,\dots,k_n=1}^l\lambda_{k_1,\dots,k_n}\prod\limits_{j=1}^n1_{\Omega_{k_j}^j}(x_j).\]
We note that 
\[\langle f,u_{(k)}\rangle =\lambda_{k_1,\dots,k_n}\prod\limits_{j=1}^n\gamma_{(k)}^{\frac{1}{p_j}}\mu_j(\Omega_{k_j}^j)^{\frac{1}{p_j}},\]
and
\[\|\langle f,u_{(k)}\rangle\|_{\ell ^{p_n}}=\left(\sum\limits_{(k)}|\lambda_{(k)}|^{p_n}\prod\limits_{j=1}^n\gamma_{(k)}^{\frac{p_n}{p_j}}\mu_j(\Omega_{k_j}^j)^{\frac{p_n}{p_j}}\right)^{\frac{1}{p_n}}.\]
We also have
\[\|f\|_{L_w^P(\mu)}=\left(\sum\limits_{k_1,\dots,k_n=1}^l|\lambda_{k_1,\dots,k_n}|^{p_n}\prod\limits_{j=1}^n
\gamma_{(k)}^{\frac{p_n}{p_j}}
\mu_j(\Omega_{k_j}^j)^{\frac{p_n}{p_j}}\right)^{\frac{1}{p_n}}.\]
Since $\|f\|_{L_w^P}\leq 1$, we have shown that $\|\langle f,u_{(k)}\rangle\|_{\ell ^{p_n}}\leq 1$. The proof for $\|\langle v_{(k)},g\rangle\|_{\ell ^{p_n'}}$ is similar.

\medskip
Therefore $\|L\|_{{\mathcal{L}(L_w^P)}}\leq 1$. The rest of the proof follows as in (Step 1).

\medskip
(Step 3) We now suppose that the weight $w$ belongs to the
mixed-norm space $L^{P}(\mu)$ for some 
$1\leq P<\infty$. 

Then there exists an increasing sequence $w_m$ of elementary weights which can be written in the form \eqref{ew1} and $\sup\limits_m w_m=w$. We set
\beq w_m(x_1,\dots,x_n)=\sum\limits_{k_1,\dots,k_n=1}^l\gamma_{(k)}^{(m)}\prod\limits_{j=1}^n1_{\Omega_{k_j}^{j,m}}(x_j),\label{ew2a}\eq
where the sets $\prod\limits_{j=1}^n\Omega_{k_j}^{j,m}$ are disjoint, $\mu_j(\Omega_{k_j}^{j,m})<\infty$ and $\gamma_{k_1,\dots,k_n}^{(m)}>0$ for all $(k)$. 

By considering \[u_{(k)}^m:=\frac{\prod\limits_{j=1}^n(\gamma_{k}^{(m)})^{\frac{1}{p_j}}1_{\Omega_{k_j}^{j,m}}}{\prod\limits_{j=1}^n\mu_j(\Omega_{k_j}^{j,m})^{\frac{1}{p_j'}}},\,\,v_{(k)}:=\frac{\prod\limits_{j=1}^n(\gamma_{k}^{(m)})^{\frac{1}{p_j'}}1_{\Omega_{k_j}^{j,m}}}{\prod\limits_{j=1}^n\mu_j(\Omega_{k_j}^{j,m})^{\frac{1}{p_j}}}, \]

\[L_m:=\sum\limits_{(k)=k_1,\dots, k_n=1}^lv_{(k)}^m\otimes u_{(k)}^m.\]
The desired operator $L$ can be obtained by defining $L:=\lim\limits_m L_m$ in $\mathcal{L}(L_w^P)$. We note that 
 since $\|L_m\|_{\mathcal{L}(L_{w_m}^P)}\leq 1$ one gets $\|L\|_{\mathcal{L}(L_{w}^P)}\leq 1$.



\medskip
(Step 4) Now if a general weight $w$ satisfies \eqref{conw1}, we have the
estimate
$$w(x_1,\dots,x_n)\leq  w_1(x_1)\cdots w_n(x_n)$$ 
for some positive functions $w_{j}$.
Let $\psi_{j}(x_{j})$ be a positive function of a variable $x_{j}$ only such that
$\frac{1}{\psi_{j}}\in L^{\widetilde{p}}(\mu_{j})$ for some $1\leq \widetilde p<\infty$.
We denote $\widetilde{w_{j}}:=w_{j}\psi_{j}$.
Then we observe that
by writing 
\[
w(x_1,\dots,x_n)=w(x_1,\dots,x_n)  \widetilde{w_1}(x_1)^{-1}\cdots 
\widetilde{w_n}(x_n)^{-1}\widetilde{w_1}(x_1)\cdots \widetilde{w_n}(x_n),
\]
we obtain
\[
\|f\|_{L_w^P(\mu)}=\|f\|_{L_{\widetilde{w}}^P(\widetilde{\mu})},
\]
where 
\[ 
\widetilde{w}=w(x_1,\dots,x_n)\widetilde{w_1}(x_1)^{-1}\cdots 
\widetilde{w_n}(x_n)^{-1}\in L^{\widetilde P}(\mu)
\] 
for $1\leq \widetilde P=(\tilde{p},\dots,\tilde{p})<\infty$
and
\[
\widetilde{\mu}=\widetilde{w_1}(x_1)\mu_1\otimes\cdots\otimes \widetilde{w_n}(x_n)\mu_n.
\]
Since $\widetilde{w}\in L^{\widetilde P}(\mu)$ for some $1\leq \widetilde P<\infty$, 
by (Step 3) applied to ${L_{\widetilde{w}}^P(\widetilde{\mu})}$,
the approximation property follows for $L_w^P(\mu)$.
\medskip
This concludes the proof of the theorem.
\end{proof}
\section{$\M_{w}^{p,q}$ and $\W_{w}^{p,q}$
have the metric approximation property}
\label{SEC:AP-mw}

It is also important to consider the special case of discrete weighted mixed-norm spaces. Given $\alpha, \beta>0$, a strictly positive function $\widetilde{w}$ on the lattice $\alpha\zet^d\times\beta\zet^d$,  we denote by $\ell_{\widetilde{w}}^{p,q}(\zet^{2d})$ the set of sequences $a=(a_{kl})_{k,l\in\zet^d}$ for which the norm 
\[\|a\|_{\ell_{\widetilde{w}}^{p,q}}=\left(\sum\limits_{l\in\zet^d}\left(\sum\limits_{k\in\zet^d}|a_{kl}|^p\widetilde{w}(\alpha k, \beta l)^p\right)^{\frac qp}\right)^{\frac 1q}\]
is finite. The main example arises from restrictions of weights on $\ard$ to a lattice $\alpha\zet^d\times\beta\zet^d$ and will be crucial in the next Corollary \ref{COR:mod}.

\medskip
Let us recall now the definition of the Wiener amalgam spaces $\W^{p,q}_{w}(\mathbb R^{d})$.
There are several definitions possible for the spaces $\W^{p,q}_{w}$, in particular
involving the short-time Fourier transform similarly to the definition of the modulation
spaces in \eqref{EQ:modul}. 
To make an analogy with modulation spaces, we can reformulate their definition
\eqref{EQ:modul} in terms of the mixed-normed Lebesgue spaces, by saying that
\begin{equation}\label{EQ:def-mod}
f\in \M^{p,q}_{w}(\mathbb R^{d}) \;\textrm{ if and only if }\;
V_{g}f\cdot w\in L^{(p,q)}(\mathbb R^{d}\times \mathbb R^{d}).
\end{equation}
Now, for a function
$F\in L^{1}_{loc}(\mathbb R^{2d})$,
we denote $\Rev F(x,\xi):=F(\xi,x).$
Then we can define
\begin{equation}\label{EQ:def-was}
f\in \W^{p,q}_{w}(\mathbb R^{d}) \;\textrm{ if and only if }\;
\Rev (V_{g}f\cdot w)\in L^{(q,p)}(\mathbb R^{d}\times \mathbb R^{d}).
\end{equation}
However, for our purposes the following description
through the Fourier transform will be more practical. 
For a review of different definitions we refer to 
\cite{Ruzhansky-Sugimoto-Toft-Tomita:MN-2011}.
So, in what follows, we will
always assume that the weights in modulation and Wiener amalgam spaces
are submultiplicative and polynomially moderate\footnote{But we do not need to assume this when
talking about weighted mixed-norm $L^{P}$-spaces.} as in
\eqref{polyw0}--\eqref{polyw}.
Then, because of the identity
$$
|V_{g}f(x,\xi)|=(2\pi)^{-d}\, |V_{\widehat{g}} \widehat{f}(\xi,-x)|,
$$
the Wiener amalgam space
$\W^{p,q}_{w}$ and the modulation spaces are related through the Fourier transform 
by the formula
\begin{equation}\label{EQ:WM}
\|f\|_{\W^{p,q}_{w}}\simeq \|\widehat{f}\|_{\M^{q,p}_{w_{0}}},
\end{equation}
 where $w(x,\xi)=w_{0}(\xi,-x).$

As a consequence of Theorem \ref{app12} we now obtain:
\begin{cor} \label{COR:mod}
Let $1\leq p,q<\infty$, and $w$ a submultiplicative polynomially moderate weight. 
Then $\M_{w}^{p,q}$ has the metric approximation property. 
Consequently, also the Wiener amalgam space $\W^{p,q}_{w}$ 
has the metric approximation property.
\end{cor} 
\begin{proof} 
We first observe that the polynomially moderate
weights satisfy conditions \eqref{conw1} and \eqref{conw2} by choosing 
$\beta_j\geq s$. Also, we have the (topological) equivalence
$\M_{w}^{p,q}\cong \ell_{\widetilde{w}}^{p,q}(\zet^{2d})$ with equivalence of norms,
by \cite[Theorems 12.2.3, 12.2.4]{groch:bK}, 
where $\widetilde{w}$ is the restriction of $w$ to the lattice $\alpha\zet^d\times\beta\zet^d$, i.e.
$\widetilde{w}(j,k)=w(\alpha j,\beta k).$
The result now follows from Theorem \ref{app12} by taking 
$\Omega_1,\dots ,\Omega_{2d}=\zet$, the weight $\widetilde{\omega}$,  
$\mu_1=\dots =\mu_{2d}=\mu$ (the counting measure). 
This proves the metric approximation property for $\M_{w}^{p,q}$ since 
that property is preserved under isomorphism. 

The metric approximation property for the Wiener amalgam spaces now follows from 
that in modulation spaces in view of
the relation \eqref{EQ:WM}.
\end{proof}

It was observed by Feichtinger and Gr\"ochenig \cite{Feichtinger-Grochenig:approx-MM-1989} that for the metric approximation property for a space it is enough to establish it for the sequence space obtained through the atomic decompositions should they exist. Our approach is however different: for the spaces $\M_{w}^{p,q}$ and $\W^{p,q}_{w}$ we have obtained it immediately as a direct consequence of the established property for the mixed-norm Lebesgue spaces. The method of proof in this paper has a certain advantage from the point of view of being applicable to spaces which are not necessarily translation invariant, see \cite{Delgado-Ruzhansky:Var-Lp} for its application to the Lebesgue spaces $L^{p(\cdot)}$ with variable exponent.

\section{$r$-nuclearity on weighted mixed-norm spaces $L_w^P$}
\label{SEC:r-nuc-Lp}

Since the metric approximation property is now established for the spaces of our interest, 
it is now relevant to consider nuclear operators on weighted mixed-norm spaces $L_w^P$.
 In this section we will characterise nuclear operators on $L_w^P$ and
present some applications to the study of the harmonic oscillator on modulation spaces $\M_w^{p,q}$. 

We first formulate a basic lemma for special measures and weights. We will consider 
$1\leq P, Q<\infty$. The multi-index $P$ will be associated to 
 the measures $\mu_i$ $(i=1,\dots, l)$ and $Q$ will correspond to the measures 
$\nu_j$ $(j=1,\dots, m)$. We will also denote $\mu:=\mu_1\otimes\cdots\otimes\mu_l$ and 
 $\nu:=\nu_1\otimes\cdots\otimes\nu_m$ the corresponding product measures on the product spaces 
$\Omega=\prod\limits_{i=1}^l\Omega_i, \Xi=\prod\limits_{j=1}^m\Xi_j$. For a weight $w$ we will 
denote $w_P(\Omega):=\|1_{\Omega}\|_{L_w^P(\mu)}$. The additional property \eqref{conw1}
will be only required for the formulation of trace relations. 
\begin{lem}\label{l1}
Let $({\Omega}_i,{\mathcal{M}}_i,\mu_i ) (i=1,\dots, l)$, 
$({\Xi}_j,{\mathcal{M}}'_j,{\nu}_j) (j=1,\dots, m)$ be measure spaces. Let $w, \widetilde{w}$ 
be weights on $\Omega, \Xi$
 respectively such that $w_P(\Omega), \widetilde{w}_{Q'}^{-1}(\Xi)<\infty$. Let $f\in L_w^{P}(\mu)$, and $(g_n)_n,(h_n)_n$ be sequences in
$L_{\widetilde{w}}^{Q}(\nu)$
  and $L_{w^{-1}}^{P'}(\mu)$, respectively, such that $\sum \limits_{n=1}^\infty \| g_n\|_{L_{\widetilde{w}}^{Q}(\nu)} \|h_n\|_{L_{w^{-1}}^{P'}(\mu)}<\infty$. Then \begin{itemize}
\item[(a)] The series $\sum\limits_{j=1}^{\infty} g_j(x)h_j(y)$
converges absolutely for a.e. $(x,y)$ and, consequently, 
${\displaystyle\lim\limits_n\sum\limits_{j=1}^n
        g_j(x)h_j(y)\,\mbox{ is finite for a.e.}\, (x,y)}.$

\item[(b)] For $k(x,y):={\displaystyle \sum\limits_{j=1}^{\infty} g_j(x)h_j(y)},$
we have $k\in L^1(\nu\otimes\mu)$.

\item[(c)] If $k_n(x,y)=\sum\limits_{j=1}^n
        g_j(x)h_j(y) $ then $\|k_n-k\|_{L^1(\nu\otimes\mu)}\rightarrow 0$.

\item[(d)] ${\displaystyle\lim\limits_n \int\limits_{\Omega}\left(\sum\limits_{j=1}^n  g_j(x)h_j(y)\right)f(y)d\mu(y)= \int\limits_{\Omega} \left(\sum\limits_{j=1}^\infty  g_j(x)h_j(y)\right)f(y)d\mu(y),}$
\end{itemize}
\noindent \textit{ for  a.e} $x$.
\end{lem}
\begin{proof} Let $k_n(x,y):=\sum\limits_{j=1}^n g_j(x)h_j(y)f(y)$. Applying the H{\"o}lder inequality we obtain 
\begin{align*}
\int\limits_{\Omega}\!\int\limits_{\Xi}|k_n(x,y)|d\nu(x)d\mu(y)\leq&\int\limits_{\Omega}\!\int\limits_{\Xi}\sum\limits_{j=1}^n
|g_j(x)
  h_j(y)f(y)|d\nu(x)d\mu(y)\\
\leq &\sum\limits_{j=1}^n \int\limits_{\Xi}
|g_j(x)|d\nu(x)\int\limits_{\Omega}|h_j(y)||f(y)|d\mu(y)\\
\leq& \widetilde{w}_{Q'}^{-1}(\Xi)\|f\|_{L_w^{P}(\mu)}\sum \limits_{j=1}^n \| g_j\|_{L_{\widetilde{w}}^{Q}(\nu)}
\|h_j\|_{L_{w^{-1}}^{P'}(\mu)}\\
\leq & M<\infty \; \mbox{for all}\; n.
\end{align*}
Hence $\|k_n\|_{{L}^1(\nu\otimes\mu)}\leq M$ for all $n$. 

On the other hand, the sequence $(s_n)$ with
 $s_n(x,y)=\sum\limits_{j=1}^n |g_j(x)h_j(y)f(y)|$, 
 is increasing in   $L^1(\nu\otimes\mu)$ and verifies
 $$
 \sup\limits_n \int\int |s_n(x,y)|d\mu(x)d\mu(y) \leq M<\infty.
 $$ Using Levi monotone convergence theorem the limit $\,s(x,y)=\lim\limits_n s_n(x,y)$
 is finite for a.e. $(x,y)$. Moreover  $s\in
 L^1(\nu\otimes\mu)$, choosing $f=1$ which belongs to $L_w^{P}(\mu)$ and from the fact that
$|k(x,y)|\leq s(x,y)$ we deduce (a) and (b). Part (c) can be deduced
using Lebesgue dominated convergence theorem applied to the sequence
$(k_n)$ dominated by $s(x,y)$. For the part (d) we observe that
letting $k_n(x,y)=\sum\limits_{j=1}^n g_j(x)h_j(y)f(y)$, we have
$|k_n(x,y)|\leq s(x,y)$ for all $n$ and every $(x,y)$. From the
fact that $s\in L^1(\nu\otimes\mu)$ we obtain that
$s(x,\cdot)\in L^1(\mu)$ for a.e $x$. Then (d) is obtained from Lebesgue dominated
convergence theorem.
\end{proof}
\begin{rem} In the case of a single measure space ($l=1$),  the condition $w_P(\Omega)<\infty$ is equivalent to the fact that $w\mu$ is a finite measure. In particular, if $w=1$ we have $w_P(\Omega)<\infty$ if and only if $\mu$ is a finite measure. 
\end{rem}
We establish below a characterisation of nuclear operators 
on the weighted mixed-norm spaces 
$L_w^{P}$ for weights satisfying the assumptions of Lemma \ref{l1}.

\begin{thm}\label{ch1} 
Let $0<r\leq 1$. Let $({\Omega}_i,{\mathcal{M}}_i,\mu_i) (i=1,\dots, l)$, 
$({\Xi}_j,{\mathcal{M}}'_j,{\nu}_j) (j=1,\dots, m)$ be measure spaces.  Let $w, \widetilde{w}$ be
weights on $\Omega, \Xi$ respectively satisfying conditions 
$w_P(\Omega), \widetilde{w}_{Q'}^{-1}(\Xi)<\infty$. 
Then $T$ is
  $r$-nuclear operator from $L_w^{P}(\mu)$ into $L_{\widetilde{w}}^{Q}(\nu)$ if and only if there exist a sequence
 $(g_n)$ in $L_{\widetilde{w}}^{Q}(\nu)$, and a sequence $(h_n)$ in $L_{w^{-1}}^{P'}(\mu)$ such that $\sum \limits_{n=1}^\infty \|
 g_n\|_{L_{\widetilde{w}}^{Q}(\nu)}^r\|h_n\|_{L_{w^{-1}}^{P'}(\mu)}^r<\infty$, and such that for all $f\in L_w^{P}(\mu)$
\[Tf(x)=\int\limits_{\Omega}\left(\sum\limits_{n=1}^{\infty}
  g_n(x)h_n(y)\right)f(y)d\mu(y), \,\,\mbox{for a.e } x.\]
\end{thm}
\begin{proof} It is enough to consider the case $r=1$. The case $0<r<1$ follows by inclusion. 
Let $T$ be a nuclear operator from $L_w^{P}(\mu)$ into $L_{\widetilde{w}}^{Q}(\nu)$. Then
there exist sequences $(g_n)$ in $L_{\widetilde{w}}^{Q}(\nu)$, $(h_n)$ in
$L_{w^{-1}}^{P'}(\mu)$ such that
 $\sum \limits_{n=1}^\infty \| g_n\|_{L_{\widetilde{w}}^{Q}(\nu)}
 \|h_n\|_{L_{w^{-1}}^{P'}(\mu)}<\infty$ and
\[Tf= \sum\limits_n \left <f,h_{n}\right>g_n .\]
\noindent Now
 \[Tf=\sum\limits_n \left <f,h_n\right>g_n =\sum\limits_n \left(\int\limits_{\Omega}
   h_n(y)f(y) d\mu(y)\right)g_n  \, ,  \]
where the sums converges in the $L_{\widetilde{w}}^{Q}(\nu)$ norm. There exists
(cf. \cite{bp:mxlp}, Theorem 1(a))
two sub-sequences $(\widetilde{g}_n)$ and $(\widetilde{h}_n)$ of $(g_n)$ and
$(h_n)$ respectively such that

\[(Tf)(x)=\sum\limits_n \left <f,\widetilde{h}_n\right>\widetilde{g}_n (x)=\sum\limits_n \left(\int\limits_{\Omega}\widetilde{h}_n(y)f(y) d\mu(y)\right)\widetilde{g}_n (x),  \,\,\mbox{ a.e } x. \]
\noindent Since the couple $\left((\widetilde{g}_n),(\widetilde{h}_n) \right)$
satisfies
\[\sum \limits_{n=1}^\infty \|\widetilde{ g}_n\|_{L_{\widetilde{w}}^{Q}(\nu)}
 \|\widetilde{h}_n\|_{L_{w^{-1}}^{P'}(\mu)}<\infty \, ,\]
  by applying Lemma \ref{l1} (d), it follows that
\begin{align*}
 \sum\limits_{n=1}^{\infty} \left(\int\limits_{\Omega}\widetilde{h}_n(y)f(y)
d\mu(y)\right)\widetilde{g}_n(x)=&\lim\limits_n  \sum\limits_{j=1}^n
\left(\int\limits_{\Omega}\widetilde{h}_j(y)f(y) d\mu(y)\right)\widetilde{g}_j(x)\\
=&\lim\limits_n \int\limits_{\Omega} \left(\sum\limits_{j=1}^n
\widetilde{g}_j(x)\widetilde{h}_j(y)f(y)\right)d\mu(y)\\
=&\int\limits_{\Omega}\left(\sum\limits_{n=1}^{\infty}
  \widetilde{g}_n(x)\widetilde{h}_n(y)\right)f(y)d\mu(y)\, , \,\mbox{ a.e } x.  \end{align*}
Conversely, assume that there exist sequences
$(g_n)_n$ in $L_{\widetilde{w}}^{Q}(\nu)$, and $(h_n)_n$ in $L_{w^{-1}}^{P'}(\mu)$ so
that $\sum \limits_{n=1}^\infty \| g_n\|_{L_{\widetilde{w}}^{Q}(\nu)}
 \|h_n\|_{L_{w^{-1}}^{P'}(\mu)}<\infty$, and for all $f\in L_w^{P}(\mu)$
\[Tf(x)=\int\limits_{\Omega}\left(\sum\limits_{n=1}^{\infty}
  g_n(x)h_n(y)\right)f(y)d\mu(y)\, , \,\,\mbox{ a.e } x.\]
By Lemma \ref{l1} (d) we have
\begin{align*}
 \int\limits_{\Omega}\left(\sum\limits_{n=1}^{\infty}
   g_n(x)h_n(y)\right)f(y)d\mu(y)&=&\lim\limits_n
 \int\limits_{\Omega}\left(\sum\limits_{j=1}^n  g_j(x)h_j(y)f(y)\right)d\mu(y)\\
  &=&\lim\limits_n  \sum\limits_{j=1}^n \left(\int\limits_{\Omega} h_j(y)f(y)
   d\mu(y)\right)g_j(x)\\
   &=& \sum\limits_n \left(\int\limits_{\Omega} h_n(y)f(y)
   d\mu(y)\right)g_n(x)\\
&=&\sum\limits_n \left <f,h_n\right>g_n
 (x)=(Tf)(x) \, ,\,\,a.e.\, x.
\end{align*}
To prove that $Tf=\sum\limits_n \left
<f,h_n\right>g_n $ in $L_{\widetilde{w}}^{Q}(\nu)$ we let $s_n:=
\sum\limits_{j=1}^n\left<f,h_j\right>g_j$, then $(s_n)_n$ is a sequence in
$L_{\widetilde{w}}^{Q}(\nu)$ and
$$|s_n(x)|\leq \|f\|_{L_w^{P}(\mu)}\sum\limits_{j=1}^n
\|h_j\|_{L_{w^{-1}}^{P'}(\mu)}|g_j(x)|$$
\[\leq\|f\|_{L_w^{P}(\mu)}\sum\limits_{j=1}^\infty\|h_j\|_{L_{w^{-1}}^{P'}(\mu)}|g_j(x)|=:\gamma(x),\, \mbox{ for all }n.\]
Moreover, $\gamma$ is well defined and $\gamma\in L_{\widetilde{w}}^{Q}(\nu)$
since it is the increasing limit of the sequence
$(\gamma_n)_n=(\|f\|_{L_w^{P}(\mu)}\sum\limits_{j=1}^n
\|h_j\|_{L_{w^{-1}}^{P'}(\mu)}|g_j(x)|)_n$ of $L_{\widetilde{w}}^{Q}(\nu)$ functions and
$$
\|\gamma_n\|_{L_{\widetilde{w}}^{Q}(\nu)}\leq\|f\|_{L_w^{P}}\sum\limits_{j=1}^\infty\|h_j\|_{L_{w^{-1}}^{P'}(\mu)}\|g_j\|_{L_{\widetilde{w}}^{Q}(\nu)}\leq M<\infty.
$$
By  the Levi monotone convergence theorem we see that
$\gamma\in L_{\widetilde{w}}^{Q}(\nu)$. Finally, applying the Lebesgue dominated
convergence theorem we deduce that $s_n\rightarrow Tf$ in $L_{\widetilde{w}}^{Q}(\nu)$.
\end{proof}
Before establishing a characterisation of $r$-nuclear operators for more general measures and weights we first generalise Lemma \ref{l1}. 
The following definition will be useful. 
\begin{defn}\label{triple} 
Let $({\Omega}_i,{\mathcal{M}}_i,\mu_i ) (i=1,\dots, l)$ be measure spaces and
$\mu:=\mu_1\otimes\cdots\otimes\mu_l$ the corresponding product measure on $\Omega=\prod\limits_{i=1}^l\Omega_i$
. We will also call $\Lambda\in\mathcal{M}:=\bigotimes\limits_{i=1}^l{\mathcal{M}}_i$ a {\em box} if it is of the form $\Lambda=\prod\limits_{i=1}^l\Lambda_i$. For a measure $\mu$, a weight $w$ on $\Omega$ and a multi-index $P$ we will say that the triple $(\mu,w, P)$ is $\sigma$-finite if there exists a family of disjoint boxes $\Omega^k$ such that $\mu(\Omega^k)<\infty$, $\bigcup\limits_{k=1}^{\infty}\Omega^k=\Omega$ and 
\[w_P(\Omega^k)=\|1_{\Omega^k}\|_{L_w^P(\mu)}<\infty.\]
\end{defn}
\begin{rem}
We observe that for the case of a single measure space ($l=1$),  a triple $(\mu,w, p)$ is 
$\sigma$-finite if and only if $w\mu$ is $\sigma$-finite. If in addition we restrict to consider weights such that $0<w(x)<\infty$ then triple $(\mu,w, p)$ is $\sigma$-finite if and only if the measure $\mu$ is $\sigma$-finite.
\end{rem}
\begin{lem}\label{l1ab}
Let $({\Omega}_i,{\mathcal{M}}_i,\mu_i ) (i=1,\dots, l)$, 
$({\Xi}_j,{\mathcal{M}}'_j,{\nu}_j) (j=1,\dots, m)$ be measure spaces. Let $1\leq P,Q<\infty$. Let $w, \widetilde{w}$ be weights on $\Omega, \Xi$ respectively such that the triples $(\mu,w, P), (\nu,{\widetilde{w}}^{-1}, Q')$ are $\sigma$-finite. Let $f\in L_w^{P}(\mu)$, and $(g_n)_n,(h_n)_n$ be sequences in
$L_{\widetilde{w}}^{Q}(\nu)$
  and $L_{w^{-1}}^{P'}(\mu)$, respectively, such that $\sum \limits_{n=1}^\infty \| g_n\|_{L_{\widetilde{w}}^{Q}(\nu)} \|h_n\|_{L_{w^{-1}}^{P'}(\mu)}<\infty$. Then
 the parts (a) and (d) of Lemma \ref{l1} hold.
\end{lem}
\begin{proof} (a) Since the triples $(\mu,w, P), (\nu,{\widetilde{w}}^{-1}, Q')$ are $\sigma$-finite, there exist two sequences
$(\Omega^k)_k$ and $(\Xi^j)_j$ of disjoint subsets of $\Omega$ and
$\Xi$ respectively such that $\bigcup_k\Omega^k=\Omega,$ $\bigcup_j\Xi^j=\Xi$ and for all $j,k$
\[w_P(\Omega^k), \widetilde{w}_{Q'}^{-1}(\Xi^j)<\infty .\]
We now consider the  measure spaces $(\Omega^k,{\mathcal
   M}^k,\mu^k)$ and  $(\Xi^j,{\mathcal
   M '}^j,\nu^j)$ that we obtain by restricting $\Omega$ to
 $\Omega^k$, and $\Xi$ to
 $\Xi^j$ for every $k,j$, and restricting the functions $g_n$
 to $\Xi^j$, and  $h_n$ to $\Omega^k$. Then, for all $k,j$
\[\sum \limits_{n=1}^\infty \| g_n\|_{L_{\widetilde{w}}^{Q}(\nu^j)} \|h_n\|_{L_{w^{-1}}^{P'}(\mu^k)}<\infty. \]
\noindent By Lemma \ref{l1} (a) it follows that 
 $\sum\limits_{j=1}^{\infty} g_j(x)h_j(y)$ converges absolutely for  a.e $(x,y)\in
\Xi^j\times \Omega^k $. Hence
 $\sum\limits_{j=1}^{\infty}
g_j(x)h_j(y)$ converges absolutely for almost every  $(x,y)\in \Xi\times \Omega $. This proves part (a).\\
\noindent From the part (a) the series
$\sum\limits_{j=1}^{\infty} g_j(x)h_j(y)f(y)$ converges absolutely
for a.e. $(x,y)\in \Xi\times\Omega$, the part (d) follows from the Lebesgue dominated
convergence theorem applied as in the ``only if" part of the proof
of Theorem \ref{ch1} (use of $\gamma_n$, and $\gamma$).
\end{proof}
We can now formulate a characterisation of $r$-nuclear operators on weighted mixed-norm spaces and a trace formula. 
\begin{thm}\label{ch2} 
Let $0<r\leq 1$. Let $({\Omega}_i,{\mathcal{M}}_i,\mu_i ) (i=1,\dots, l)$, 
$({\Xi}_j,{\mathcal{M}}'_j,{\nu}_j) (j=1,\dots, m)$ be measure spaces. Let $1\leq P,Q<\infty$. 
Let $w, \widetilde{w}$ be weights on $\Omega, \Xi$ respectively such that the triples 
$(\mu,w, P), (\nu,{\widetilde{w}}^{-1}, Q')$ are $\sigma$-finite . Then $T$ is
  $r$-nuclear operator from $L_w^{P}(\mu)$ into $L_{\widetilde{w}}^{Q}(\nu)$ if and only if there exist a sequence
 $(g_n)$ in $L_{\widetilde{w}}^{Q}(\nu)$, and a sequence $(h_n)$ in $L_{w^{-1}}^{P'}(\mu)$ such that $\sum \limits_{n=1}^\infty \|
 g_n\|_{L_{\widetilde{w}}^{Q}(\nu)}^r\|h_n\|_{L_{w^{-1}}^{P'}(\mu)}^r<\infty$, and such that for all $f\in L_w^{P}(\mu)$
\[Tf(x)=\int\limits_{\Omega}\left(\sum\limits_{n=1}^{\infty}
  g_n(x)h_n(y)\right)f(y)d\mu(y), \,\,\mbox{for a.e } x.\]
Moreover, if $w=\widetilde{w}$ satisfies \eqref{conw1}, $\mu=\nu$, $P=Q$ and $T$ is $r$-nuclear on $\mathcal{L}(L_w^{P}(\mu))$ with $r\leq \frac 23$, then  
\[\Tr(T)=\sum\limits_{j=1}^{\infty}\lambda_j,\]
where $\lambda_j\,\, (j=1,2,\dots)$ are the eigenvalues of $T$ with multiplicities taken into account,
and $\Tr(T)=\sum\limits_{j=1}^{\infty} \left< u_{j},v_{j} \right>.$
\end{thm}

\begin{proof} Again, for the proof of the characterisation it is enough to consider the case $r=1$. But that characterisation
 now follows from the same lines of the proof of Theorem \ref{ch1} by replacing the references to part (d) of Lemma \ref{l1} by the part (d) of Lemma \ref{l1ab}.  On the other hand, since $w$ additionally satisfies \eqref{conw1} the metric approximation property holds. If 
  $T$ is $r$-nuclear with $r\leq \frac 23$ on $L_w^P$, the trace 
  formula follows from the aforementioned 
 Grothendieck's theorem in the introduction.
\end{proof}

\section{$r$-nuclearity on modulation spaces and the harmonic oscillator}
\label{SEC:r-nuc-mod}

In this section we describe the $r$-nuclearity and a trace formula in modulation spaces.
We restrict our attention to modulation spaces but note that the same conclusions 
hold also in the Wiener amalgam spaces.
Thus, as an immediate consequence of Theorem \ref{ch2} and Corollary \ref{COR:mod}
we have:
\begin{cor} \label{COR:M-gen}
Let $0<r\leq 1$, $1\leq p,q<\infty$ and $w$ a submultiplicative polynomially
moderate weight. 
An operator $T\in\mathcal{L}(\M_w^{p,q},\M_w^{p,q})$ is $r$-nuclear if and only if 
its kernel $k(x,y)$ can be written in the form
\[k(x,y)=\sum\limits_{j=1}^{\infty}u_j\otimes v_j,\]
with $u_j\in\M_w^{p,q},v_j\in\M_{w^{-1}}^{p',q'} $ and
\[\sum\limits_{j=1}^{\infty}\|u_j\|_{\M_w^{p,q}}^r\|v_j\|_{\M_{w^{-1}}^{p',q'}}^r<\infty.\]
Moreover, if $T$ is $r$-nuclear with $r\leq \frac 23$, then  
\begin{equation}\label{EQ:mod-trace}
\Tr(T)=\sum\limits_{j=1}^{\infty}\lambda_j,
\end{equation}
where $\lambda_j\,\, (j=1,2,\dots)$ are the eigenvalues of $T$ with multiplicities taken into account.
\end{cor}

In principle, in Corollary \ref{COR:M-gen},
the order $r\leq\frac23$ in the $r$-nuclearity is sharp for the validity of
the trace formula \eqref{EQ:mod-trace} in the context of general Banach spaces
(with approximation property). However, for the traces in, for example, 
the $L^{p}$-spaces the trace formula \eqref{EQ:mod-trace} may hold for
$r$-nuclear operators also with larger values of $r$. In fact, the condition $r\leq \frac23$
may be relaxed to the condition $r\leq r_{0}(p)$ with the index $r_{0}(p)\geq \frac23$ depending
on $p$, see Reinov and Laif \cite{Reinov}, as well as the authors' paper
\cite{dr13a:nuclp}. The same property may be expected also for general
operators in $L^P$ weighted mixed-norm spaces and
 consequently in modulation spaces.

However, for some special operators the trace formula \eqref{EQ:mod-trace} may be valid
for even larger values of $r$, for example even for simply nuclear operators (i.e. for
$r=1$).

\medskip
We shall now consider an application of such nuclearity concepts 
to the study of the harmonic oscillator 
$A=-\Delta +|x|^2$ on $\arda$. We will consider in particular the modulation space $\M_s^{p,q}$ corresponding to the weight $w(x,\xi)=(1+|\xi|)^s$. 
If
$$A\phi _j\equiv (-\Delta +|x|^2)\phi _j=\lambda_j\phi _j,$$ 
the eigenvalues $\lambda_{j}$ can be enumerated in the form
$\lambda=\lambda_{(k)}=\sum_{i=1}^{d} (2k_{i}+1),$
$k=(k_{1},\ldots, k_{d})\in\mathbb N^{d}$, see e.g.
\cite[Theorem 2.2.3]{Nicola-Rodino:bk}.
For the corresponding sequence of orthonormal eigenfunctions $\phi_{j}$
in $L^2(\arda)$, we can write with convergence in $L^2(\arda)$:
\[f=\sum\limits_{j=1}^{\infty}\langle f,\overline{\phi _j}\rangle\phi _j.\]
Hence, formally the kernel of $A$ can be written as
$$
k(x,y)=\sum\limits_{j=1}^{\infty}A\phi _j(x)\overline{\phi _j(y)}
=\sum\limits_{j=1}^{\infty}\lambda_j\phi _j(x)\overline{\phi _j(y)}.
$$
We note that this can be justified by taking negative powers of the harmonic oscillator 
$(-\Delta +|x|^2)^{-N}$ for $N>0$ large enough so that we start with the 
decomposition for the corresponding kernel in the form
\[
k_N(x,y)=\sum\limits_{j=1}^{\infty}\lambda_j^{-N}\phi _j(x)\overline{\phi _j(y)}.
\]
More generally for functions of the harmonic oscillator, defined by 
\begin{equation}\label{EQ:HA-F}
F(-\Delta +|x|^2)\phi _j= F(\lambda_j)\phi _j, \qquad
j=1,2,\ldots,
\end{equation}
we have:
\begin{thm} 
Let $0<r\leq 1$, $s\in\er$ and $1\leq p,q<\infty$. 
The operator $F(-\Delta +|x|^2)$ is $r$-nuclear on $\M_s^{p,q}(\arda)$ provided that
\begin{equation}\label{EQ:nucl}
\sum\limits_{j=1}^{\infty} |F(\lambda_j)|^r\|\phi _j
\|_{\M_s^{p,q}}^r\|\phi _j\|_{\M_{-s}^{p',q'}}^r<\infty.
\end{equation}
Moreover, if \eqref{EQ:nucl} holds with $r=1$,
we have the trace formula
\begin{equation}\label{EQ:F-trace}
\Tr F(-\Delta +|x|^2) =\sum_{j=1}^{\infty} F(\lambda_{j}),
\end{equation}
with the absolutely convergent series.
\end{thm}

\begin{proof}
The first part follows from Corollary \ref{COR:M-gen}.
Moreover, while formula \eqref{EQ:F-trace} can be expected 
from the general Grothendieck's
theory (at least for $r\leq\frac23$), 
in this case it follows in an elementary way due to the smoothness property
of Hermite functions. Indeed, from \eqref{EQ:HA-F},
the integral kernel of $F(-\Delta +|x|^2)$ is given by 
\begin{equation}\label{EQ:F-kernel}
k(x,y)=\sum\limits_{j=1}^{\infty} F(\lambda_j)\phi _j(x)\overline{\phi _j(y)}.
\end{equation}
At the same time, we know that functions $\phi_{j}$ are all smooth and fast decaying, implying in particular
that $\phi _j\in {\M_s^{p,q}}$ (actually,
this also follows if the assumption \eqref{EQ:nucl} holds with $r=1$).
Consequently, by \eqref{EQ:Trace}, we obtain
$$
 \Tr F(-\Delta +|x|^2) =\sum_{j=1}^{\infty} F(\lambda_{j}) 
 \langle\phi_{j},\overline{\phi_{j}}\rangle_{\M_s^{p,q},\M_{-s}^{p',q'}}
 =\sum_{j=1}^{\infty} F(\lambda_{j}),
$$
in view of the equality
$$
 \langle\phi_{j},\overline{\phi_{j}}\rangle_{\M_s^{p,q},\M_{-s}^{p',q'}}=
 (\phi_{j},{\phi_{j}})_{L^{2}}=1.
$$
The series in \eqref{EQ:F-trace} converges absolutely in view of
\begin{multline*}
\sum_{j=1}^{\infty} |F(\lambda_{j})|=
\sum_{j=1}^{\infty} |F(\lambda_{j})|
 (\phi_{j},{\phi_{j}})_{L^{2}}
 =
 \sum_{j=1}^{\infty} |F(\lambda_{j})|
 \langle\phi_{j},\overline{\phi_{j}}\rangle_{\M_s^{p,q},\M_{-s}^{p',q'}} \\
 \leq
 \sum\limits_{j=1}^{\infty} |F(\lambda_j)|\|\phi _j
\|_{\M_s^{p,q}}\|\phi _j\|_{\M_{-s}^{p',q'}}<\infty,
\end{multline*}
which is finite by the assumption.
This completes the proof. 
\end{proof}


\end{document}